\documentclass[11pt]{amsart}
\usepackage{tabularx,booktabs,tikz}
\usepackage{caption}
\usepackage{amsmath}
\usepackage{amsfonts}

\usepackage{amscd}
\usepackage{amsthm}
\usepackage{amssymb} 
\usepackage{latexsym}
\usepackage{eufrak}
\usepackage{euscript}
\usepackage{epsfig}
\usepackage{graphics}
\usepackage{array}
\usepackage{enumerate}
\usepackage{dsfont}
\usepackage{color}
\usepackage{wasysym}
\usepackage{hyperref}
\usepackage{natbib}
\usepackage{pdfsync}

\newcommand{\Hmm}[1]{\leavevmode{\marginpar{\tiny%
$\hbox to 0mm{\hspace*{-0.5mm}$\leftarrow$\hss}%
\vcenter{\vrule depth 0.1mm height 0.1mm width \the\marginparwidth}%
\hbox to
0mm{\hss$\rightarrow$\hspace*{-0.5mm}}$\\\relax\raggedright #1}}}

\newtheorem{theorem}{Theorem}[section]
\newtheorem{lemma}[theorem]{Lemma}

\newtheorem{definition}[theorem]{Definition}

\newtheorem{proposition}[theorem]{Proposition}

\usepackage[margin=1.15in]{geometry}
\usepackage{amsmath,amssymb,amsthm,mathtools}
\numberwithin{equation}{section}

\theoremstyle{definition}

\title[Stable Liouville solutions on graphs]
{Nonexistence of Stable Solutions to the Liouville Equation on Graphs with Polynomial Volume Growth}

\author{Weiqi Guan}
\address{School of Mathematical Sciences, Fudan University, Shanghai 200433, China}
\email{wqguan24@m.fudan.edu.cn}



\begin{document}

\begin{abstract}
We study stable solutions of the discrete Liouville equation
\[
-\Delta u=e^u
\]
on infinite, locally finite graphs with polynomial volume growth.
We prove that no stable solution exists whenever the order $d$
of polynomial volume growth satisfies $d<10$.
\end{abstract}

\maketitle

\bibliographystyle{abbrv}

\section{Introduction}

The Liouville equation
\begin{equation}\label{eq:liouville}
    -\Delta u=e^u
\end{equation}
is a classical semilinear elliptic equation whose entire solutions have been
studied extensively in the continuous setting \cite{chen1991classification,dancer2009classification}. A central question is whether
stable solutions can exist on the whole space. This question is studied for Lane--Emden equation \cite{farina2007classification} and for general nonlinearty \cite{dupaigne2010stable} on the Euclidean space. For the Liouville equation on the Euclidean space,
Farina proved that no stable classical solution exists in $\mathbb{R}^N$ when
$2\leq N\leq 9$, and remarked that this dimensional threshold is sharp: radial stable
solutions exist when $N\geq 10$ \cite{farina2007stable}. This result makes the number
$10$ a natural threshold to investigate in discrete settings as well.

Equations on graphs are studied extensively recently \cite{grigor2016kazdan,huang2020existence}. Though, the classification of semilinear elliptic equation are studied recently \cite{gu2023semi,hao2022sharp}, the classification of stable solution to semilinear elliptic on graphs is less studied. On graphs the lack of a differential chain rule makes it difficult to transfer
arguments from Euclidean space directly. For power nonlinearities, however,
progress has been made for the discrete Lane--Emden equation
$-\Delta u=|u|^{p-1}u$. In particular, algebraic inequalities can replace key
chain-rule estimates and, together with stability and cutoff arguments, yield
nonexistence results on graphs with polynomial volume growth in ranges related
to the Joseph--Lundgren exponent \cite{guan2026nonexistence}. These estimates are adapted
to power functions and do not directly cover the exponential nonlinearity in
\eqref{eq:liouville}.

The obstruction is not merely a missing technical detail. For example, there
is no constant $C>0$ such that
\[
    (a-b)^2(e^a+e^b)\leq C(a-b)(e^a-e^b)
    \qquad\text{for all }a,b\in\mathbb{R}.
\]
Consequently, the algebraic inequality used for the discrete power equation
cannot simply be transferred to the exponential equation. Moreover, the
continuous proof of Farina relies on differential identities for $e^{\alpha u}$
and the ordinary chain rule, neither of which is available in the same form on
a graph.

In this paper, we overcome this difficulty by introducing
$v=e^{\alpha u}$ for a parameter $\alpha>0$. The convexity of the exponential
function gives the one-sided inequality
\[
    -\Delta v\leq \alpha v^{1+1/\alpha},
\]
while stability of $u$ yields a inequality for $v$. We can therefore apply suitable power-type algebraic
inequalities to $v$ and deduce suitable estimate.

Our main result is the following graph analogue of Farina's nonexistence
result.

\begin{theorem}\label{thm:main}
Let $G$ be an infinite, locally finite, connected, simple, unweighted graph
with polynomial volume growth of order $d$ in the sense of Definition~\ref{def:volume-growth}.
If $d<10$, and $u$ is a stable solution to \eqref{eq:liouville}. Then $u\equiv 0$.
\end{theorem}

The proof combines the stability estimate with cutoff functions adapted to the
polynomial volume-growth assumption. The parameters are chosen so that the
resulting exponent of the cutoff error is strictly larger than $d$; letting the
cutoff radius tend to infinity then contradicts positivity of $v$. The theorem
establishes nonexistence for growth order $d<10$. The existence or nonexistence of stable solutions when $d \geq 10$ is still open.

The paper is organized as follows. Section~2 collects the graph notation,
energy identity, stability condition, and volume-growth assumption.
Section~3 establishes the algebraic inequalities and the exponential change of
variables. Section~4 proves the key stability estimate, and Section~5 completes
the proof of Theorem~\ref{thm:main} using cutoff functions.

\section{Preliminaries}\label{sec:preliminaries}

Let $G=(V,E_G)$ be an infinite, locally finite, connected, simple, unweighted
graph. We write $x\sim y$ when vertices $x,y\in V$ are adjacent, and denote by
$B_R(x)$ the ball of radius $R$ centered at $x$ in the graph distance. Let
$C(V)$ be the space of real-valued functions on $V$, and let $C_c(V)$ denote
the subspace of finitely supported functions.

The graph Laplacian is defined by
\[
    \Delta f(x)=\sum_{y\sim x}\bigl(f(y)-f(x)\bigr),
    \qquad x\in V.
\]
For functions $f,g\in C(V)$, with at least one of them finitely supported, set
\begin{equation}\label{eq:energy-form}
    \Gamma(f,g)=\frac12\sum_{x}\sum_{y\sim x}
    (f(x)-f(y))(g(x)-g(y)).
\end{equation}
The discrete integration-by-parts formula states that
\begin{equation}\label{eq:discrete-ibp}
    \sum_{x\in V}(-\Delta f)(x)g(x)=\Gamma(f,g),
\end{equation}
whenever $g\in C_c(V)$. We also write
\[
    |\nabla f|^2(x)=\frac12\sum_{y\sim x}|f(y)-f(x)|^2.
\]
In particular, $\Gamma(f,f)=\sum_{x}|\nabla f|^2(x)$ for finitely supported
$f$.

A solution $u\colon V\to\mathbb{R}$ of \eqref{eq:liouville} is called stable if
\begin{equation}\label{eq:stability}
    \sum_{x\in V}e^{u(x)}\varphi(x)^2\leq \Gamma(\varphi,\varphi)
    \qquad\text{for every }\varphi\in C_c(V).
\end{equation}
All sums in \eqref{eq:stability} are finite for finitely supported test
functions.

\begin{definition}\label{def:volume-growth}
We say that $G$ has polynomial volume growth of order $d\geq 0$ if there is a
constant $C_0>0$ such that
\[
    |B_R(x)|\leq C_0R^d
    \qquad\text{for every }x\in V\text{ and }R\geq 1.
\]
\end{definition}

\section{Power inequalities and the exponential change of variables}
The following inequalities are the algebraic estimates used for discrete Lane--Emden equations; we include a proof for completeness.

\begin{lemma}\label{lem:algebra}
Let $a,b>0$ and $\gamma\geq1$. Then
\begin{align}
\bigl(a^{(\gamma+1)/2}-b^{(\gamma+1)/2}\bigr)^2
&\leq \frac{(\gamma+1)^2}{4\gamma}(a-b)(a^\gamma-b^\gamma), \label{eq:alg1}\\
(a-b)^2(a^{\gamma-1}+b^{\gamma-1})
&\leq 2(a-b)(a^\gamma-b^\gamma). \label{eq:alg2}
\end{align}
In both inequalities, the products on the right are nonnegative.
\end{lemma}

\begin{proof}
For \eqref{eq:alg1}, assume without loss of generality that $a\geq b$. By the Cauchy--Schwarz inequality,
\[
\begin{aligned}
\bigl(a^{(\gamma+1)/2}-b^{(\gamma+1)/2}\bigr)^2
&=\frac{(\gamma+1)^2}{4}
\left(\int_b^a t^{(\gamma-1)/2}\,dt\right)^2\\
&\leq \frac{(\gamma+1)^2}{4}(a-b)\int_b^a t^{\gamma-1}\,dt\\
&=\frac{(\gamma+1)^2}{4\gamma}(a-b)(a^\gamma-b^\gamma).
\end{aligned}
\]
For \eqref{eq:alg2}, again take $a\geq b$. Direct calculation gives
\[
a^\gamma-b^\gamma-\frac{a-b}{2}(a^{\gamma-1}+b^{\gamma-1})
=\frac{a+b}{2}(a^{\gamma-1}-b^{\gamma-1})\geq0.
\]
Multiplication by $2(a-b)$ gives \eqref{eq:alg2}.
\end{proof}
The following lemma allows us to construct a subsolution to a Lane–Emden equation from a solution of the Liouville equation.
\begin{lemma}\label{lem:transform}
Let $u$ solve \eqref{eq:liouville}, let $\alpha>0$, and set $v=e^{\alpha u}$. Then
\begin{equation}\label{eq:vsub}
-\Delta v\leq \alpha v^{1+1/\alpha}\quad\text{on }V.
\end{equation}
Moreover, stability of $u$ is equivalent to the inequality
\begin{equation}\label{eq:vstable}
\sum_{x\in V}v(x)^{1/\alpha}\varphi(x)^2\leq \Gamma(\varphi,\varphi)
\qquad(\forall\varphi\in C_c(V)).
\end{equation}
\end{lemma}

\begin{proof}
Convexity of $t\mapsto e^{\alpha t}$ implies, for every pair $x\sim y$,
\[
e^{\alpha u(y)}\geq e^{\alpha u(x)}+\alpha e^{\alpha u(x)}(u(y)-u(x)).
\]
Summing over $y\sim x$ gives $\Delta v(x)\geq\alpha e^{\alpha u(x)}\Delta u(x)$. Since $\Delta u=-e^u$,
\[
-\Delta v(x)\leq\alpha e^{(\alpha+1)u(x)}=\alpha v(x)^{1+1/\alpha},
\]
which proves \eqref{eq:vsub}. The stability inequality \eqref{eq:vstable} follows from $e^u=v^{1/\alpha}$ and \eqref{eq:stability}.
\end{proof}

\section{A stability estimate}
For $\eta\in C_c(V)$, define
\[
B_{\gamma}(v,\eta):=\frac12\sum_{x\in V}\sum_{y\sim x}v(x)^{\gamma+1}(\eta(x)-\eta(y))^2
=\sum_{x\in V}v(x)^{\gamma+1}|\nabla\eta|^2(x).
\]
We next give a crucial estimate.
\begin{proposition}\label{prop:stability-est}
Assume $u$ is a stable solution of \eqref{eq:liouville} on a graph of uniformly bounded degree. Let $\gamma\geq1$ and $\alpha>0$ satisfy
\begin{equation}\label{eq:parameter}
\alpha\frac{(\gamma+1)^2}{4\gamma}<1.
\end{equation}
Then there exists $C_{\alpha,\gamma}>0$ such that, for every nonnegative $\eta\in C_c(V)$,
\begin{equation}\label{eq:stab-est}
\sum_{x\in V}v(x)^{\gamma+1+1/\alpha}\eta(x)^2
\leq C_{\alpha,\gamma} B_{\gamma}(v,\eta),
\end{equation}
where $v=e^{\alpha u}$. The constant is independent of $\eta$.
\end{proposition}

\begin{proof}
Abbreviate
\[
A_\eta=\sum_{x\in V} v(x)^{\gamma+1+\frac1\alpha}\eta(x)^2,
\qquad
T_\eta=\frac12\sum_{x\in V}\sum_{y\sim x}
(v(x)-v(y))(v(x)^\gamma-v(y)^\gamma)\eta(x)^2.
\]
Since $v>0$ and $t\mapsto t^\gamma$ is increasing, $T_\eta\geq0$.

First, use \eqref{eq:vstable} with $\varphi=v^{(\gamma+1)/2}\eta$. For every $\delta>0$, the elementary inequality $(A+B)^2\leq(1+\delta)A^2+C_\delta B^2$ applied edgewise gives
\begin{align*}
    A_\eta&\leq \frac{1}{2}\sum_{x\in V}\sum_{y\sim x}
\left(v(x)^{(\gamma+1)/2}\eta(x)-v(y)^{(\gamma+1)/2}\eta(y)\right)^2\\
&=\frac{1}{2}\sum_{x\in V}\sum_{y\sim x}\left((v(x)^{(\gamma+1)/2}-v(y)^{(\gamma+1)/2})\eta(x)+v(y)^{(\gamma+1)/2}(\eta(x)-\eta(y))\right)^2
    \\&\leq (1+\delta)\frac12\sum_{x\in V}\sum_{y\sim x}
\left(v(x)^{(\gamma+1)/2}-v(y)^{(\gamma+1)/2}\right)^2\eta(x)^2
+C_\delta B_\gamma(v,\eta).
\end{align*}
Using \eqref{eq:alg1}, we infer
\begin{equation}\label{eq:A-T}
A_\eta\leq (1+\delta)\frac{(\gamma+1)^2}{4\gamma} T_\eta+C_\delta B_\gamma(v,\eta).
\end{equation}

Next, test \eqref{eq:vsub} against the nonnegative finitely supported function $v^\gamma\eta^2$ and sum over graph,
\begin{equation}\label{eq:test-sub}
\Gamma(v,v^\gamma\eta^2)\leq \alpha A_\eta.
\end{equation}
Expanding the edge sum yields
\begin{align*}
    \Gamma(v,v^\gamma\eta^2) &=\frac{1}{2}\sum_{x\in V}\sum_{y\sim x}(v(y)-v(x))(v(y)^{\gamma}\eta(y)^2-v(x)^{\gamma}\eta(x)^2)\\
    &=\frac{1}{2}\sum_{x\in V}\sum_{y\sim x}(v(y)-v(x))(v(y)^{\gamma}-v(x)^{\gamma})\eta(x)^2\\
    &\quad+\frac{1}{2}\sum_{x\in V}\sum_{y\sim x}(v(y)-v(x))(\eta(y)^2-\eta(x)^2)v(y)^{\gamma}\\
    & = T_\eta+S_\eta,
\end{align*}
where
\begin{align*}
    S_\eta&=\frac12\sum_{x\in V}\sum_{y\sim x}
(v(x)-v(y))v(y)^\gamma\bigl(\eta(x)^2-\eta(y)^2\bigr)\\
&=\frac12\sum_{y\sim x}
(v(x)-v(y))(v(y)^\gamma+v(x)^\gamma)\bigl(\eta(x)^2-\eta(y)^2\bigr).
\end{align*}
We claim that for every $\varepsilon\in(0,1)$ there is a constant $C_{\varepsilon}$ such that
\begin{equation}\label{eq:cross}
|S_\eta|\leq\varepsilon T_\eta+C_{\varepsilon} B_\gamma(v,\eta).
\end{equation}
Indeed, edgewise Young's inequality and \eqref{eq:alg2} gives, for any small $\varepsilon>0$,
\[
\begin{aligned}
|S_\eta|&\leq\frac12\sum_{y\sim x}\vert v(x)-v(y)|(v(y)^\gamma+v(x)^\gamma)|\eta(x)^2-\eta(y)^2|\\
&\leq \frac12\sum_{y\sim x}\varepsilon (v(x)-v(y))^2(v(y)^{\gamma-1}+v(x)^{\gamma-1}))
          (\eta(x)^2+\eta(y)^2)
       +\frac12\sum_{y\sim x}C_\varepsilon (v(y)^{\gamma+1}+v(x)^{\gamma+1})(\eta(x)-\eta(y))^2\\
&\leq \sum_{y\sim x}\varepsilon (v(x)-v(y))(v(y)^{\gamma}-v(x)^{\gamma})
          (\eta(x)^2+\eta(y)^2)
       +\frac12\sum_{y\sim x}C_\varepsilon (v(y)^{\gamma+1}+v(x)^{\gamma+1})(\eta(x)-\eta(y))^2.\\
& = \varepsilon T_{\eta}+C_{\varepsilon}B(v,\eta)
\end{aligned}
\]

Combining \eqref{eq:test-sub} and \eqref{eq:cross} gives
\begin{equation}\label{eq:T-A}
(1-\varepsilon)T_\eta\leq \alpha A_\eta+C_{\varepsilon} B_\gamma(v,\eta).
\end{equation}
Substitute \eqref{eq:T-A} into \eqref{eq:A-T}:
\[
A_\eta\leq \frac{(1+\delta)\alpha}{1-\varepsilon}\frac{(\gamma+1)^2}{4\gamma}            A_\eta
+C_{\alpha,\gamma,\delta,\varepsilon}B_\gamma(v,\eta).
\]
Because $\alpha \frac{(\gamma+1)^2}{4\gamma}<1$ by \eqref{eq:parameter}, we may choose $\delta$ and $\varepsilon$ sufficiently small so that $\frac{(1+\delta)\alpha}{1-\varepsilon}\frac{(\gamma+1)^2}{4\gamma}<1$. Hence we have 
\[
\bigl(1-\frac{(1+\delta)\alpha}{1-\varepsilon}\frac{(\gamma+1)^2}{4\gamma}\bigr)A_{\eta}\leq C_{\alpha,\gamma,\delta,\varepsilon}B_\gamma(v,\eta).
\]
This proves \eqref{eq:stab-est}.
\end{proof}

\section{Cutoffs and proof of the main theorem}
In this section, we prove the main results Theorem \ref{thm:main}.
\begin{proof}[Proof of Theorem \ref{thm:main}]
Assume for contradiction that $u$ is a non-trivial stable solution. Let $o\in V$ be fixed. For every $R\geq1$, choose a cutoff $\psi_R\colon V\to[0,1]$ satisfying
\begin{equation*}
    \psi_R(x)=
    \begin{cases}
        1, & \text{if } d(x,o)\leq R,\\[5pt]
        \dfrac{2R-d(x,o)}{R}, & \text{if } R<d(x,o)<2R,\\[8pt]
        0, & \text{if } d(x,o)\geq 2R.
    \end{cases}
\end{equation*}
Then it satisfies
\[
\psi_R=1\text{ on }B_R(o),\qquad
\psi_R=0\text{ on }V\setminus B_{2R}(o),\qquad
|\psi_R(x)-\psi_R(y)|\leq R^{-1}\quad(x\sim y).
\]
The support of $|\nabla\psi_R|$ is contained in $B_{2R+1}(o)$, whose cardinality is bounded by $CR^d$; hence, by the polynomial volume growth assumption, for every $q\geq2$,
\begin{equation}\label{eq:cutoff-gradient}
\sum_{x\in V}|\nabla\psi_R|^q(x)\leq C_q R^{d-q}.
\end{equation}
Fix $\gamma\geq1$ and $\alpha>0$ satisfying \eqref{eq:parameter}; both will be chosen below. Put
\[
\tau=2\bigl(1+\alpha(\gamma+1)\bigr).
\]
Choose an integer $k$ large enough such that
\[
(2k-2)\frac{\tau}{\tau-2}\geq2k
\]
and 
\[
2k\geq\tau .
\]
Set 
\[
A_R=\sum_{x\in V} v(x)^{\gamma+1+\frac{1}{\alpha}}\psi_R(x)^{2k}.
\]
We estimate the right-hand side of \eqref{eq:stab-est}. For $a,b\in[0,1]$ and integer $k\geq1$, the elementary inequality
\begin{equation}\label{eq:power-cutoff}
|a^k-b^k|^2\leq C_k|a-b|^2a^{2k-2}+C_k|a-b|^{2k}
\end{equation}
holds. To see this, write $\max\{a,b\}\leq a+|a-b|$ and use
$|a^k-b^k|\leq k|a-b|\max\{a,b\}^{k-1}$. Also, taking $\varphi=\delta_x$ in \ref{eq:vstable}, we have $M=\|v\|_{\ell^\infty(V)}<\infty$. \eqref{eq:power-cutoff} imply
\begin{equation}\label{eq:B-cutoff}
B_\gamma(v,\psi_R^k)
\leq C_k\sum_{x\in V}v(x)^{\gamma+1}\psi_R(x)^{2k-2}|\nabla\psi_R|^2(x)
+C_kM^{\gamma+1}\sum_{x\sim y}|\psi_R(x)-\psi_R(y)|^{2k}.
\end{equation}

Since $(2k-2)\frac{\tau}{\tau-2}\geq2k$. Thus, for any $\varepsilon>0$, Young's inequality gives pointwise
\[
v(x)^{\gamma+1}\psi_R(x)^{2k-2}|\nabla\psi_R|^2(x)
\leq \varepsilon v(x)^{\gamma+1+\frac{1}{\alpha}}\psi_R(x)^{2k}
+C_\varepsilon |\nabla\psi_R|^{\tau}(x).
\]
After summation and use of \eqref{eq:cutoff-gradient},
\[
\sum_{x\in V}v(x)^{\gamma+1}\psi_R(x)^{2k-2}|\nabla\psi_R|^2(x)
\leq \varepsilon A_R+C_\varepsilon R^{d-\tau}.
\]
Also, the number of edges meeting $B_{2R+1}(o)$ is at most $CR^d$ by the polynomial volume growth. Since $|\psi_R(x)-\psi_R(y)|\leq R^{-1}$ and $2k\geq\tau$,
\[
\sum_{x\sim y}|\psi_R(x)-\psi_R(y)|^{2k}
\leq CR^{d-2k}\leq CR^{d-\tau}.
\]
Consequently, \eqref{eq:B-cutoff} gives
\[
B_\gamma(v,\psi_R^k)\leq C_k\varepsilon A_R+C_{\varepsilon,k}R^{d-\tau}.
\]
Applying Proposition \ref{prop:stability-est} we have
\[
   A_R
\leq C_{\alpha,\gamma} B_{\gamma}(v,\psi_R^k)\leq \varepsilon C_{\alpha,\gamma}C_k A_R+C_{\alpha,\gamma}C_{\varepsilon,k}R^{d-\tau}.
\]
Choosing $\varepsilon$ small enough such that $\varepsilon C_{\alpha,\gamma}C_k<1$ to absorb the $A_R$ term, we obtain
\begin{equation}\label{eq:final-est}
A_R\leq CR^{d-\tau},
\end{equation}
where $C$ is independent of $R$.

It remains to choose the parameters. Condition \eqref{eq:parameter} is
\[
0<\alpha<\frac{4\gamma}{(\gamma+1)^2}.
\]
Hence, $\tau=2\bigl(1+\alpha(\gamma+1)\bigr)<2\bigl(1+\frac{4\gamma}{\gamma+1}\bigr)$.
Since $2\bigl(1+\frac{4\gamma}{\gamma+1}\bigr)<10$, $\lim\limits_{\gamma\rightarrow+\infty}2\bigl(1+\frac{4\gamma}{\gamma+1}\bigr)=10$ and $d<10$, we can first choose $\gamma\geq1$ sufficiently large and then choose $\alpha$ satisfying
\[
0<\alpha<\frac{4\gamma}{(\gamma+1)^2},
\qquad d<2\bigl(1+\alpha(\gamma+1)\bigr)=\tau.
\]
With this choice, $d-\tau<0$. For any fixed $R_0>0$, since $\psi_R\equiv1$ on $B_R(o)$, estimate \eqref{eq:final-est} yields, for all $R\geq R_0$,
\[
0<\sum_{x\in B_{R_0}(o)}v(x)^{\gamma+1+\frac{1}{\alpha}}
\leq A_R\leq CR^{d-\tau}\longrightarrow0
\qquad\text{as }R\to\infty.
\]
This is a contradiction. Theorem \ref{thm:main} follows.
\end{proof}

\section*{Acknowledgements}
The author is deeply grateful to Professor Bobo Hua for helpful discussions and constant support.

\section*{Declaration on the Use of AI Tools}
The author used ChatGPT (GPT-6; OpenAI, available at https://chatgpt.com/) on [2026/10/15-1016/10/30] solely to suggest grammatical, stylistic, and clarity improvements to selected sentences originally written by the author. The mathematical ideas and main results of this paper were developed by the author independently. The author takes full responsibility for the accuracy, originality, and integrity of the final manuscript, including all mathematical statements, proofs, and references.


\end{document}